\documentclass[10pt]{amsart}
\usepackage{amsfonts}
\usepackage{amsthm}
\usepackage{amsmath}
\usepackage{graphicx}
\usepackage{amscd,amssymb,amsthm}
\usepackage{epstopdf}

\title[]{Inequalities for the one-dimensional analogous of the Coulomb potential}

\author[]{\'Arp\'ad Baricz}
\address{Institute of Applied Mathematics, John von Neumann Faculty of Informatics, \'Obuda University, 1034 Budapest, Hungary}
\address{Department of Economics, Babe\c{s}-Bolyai University, 400591 Cluj-Napoca, Romania}
\email{arpad.baricz@econ.ubbcluj.ro}

\author[]{Tibor K. Pog\'any}
\address{Institute of Applied Mathematics, John von Neumann Faculty of Informatics, \'Obuda University, 1034 Budapest, Hungary}
\address{Faculty of Maritime Studies, University of Rijeka, 51000 Rijeka, Croatia}
\email{poganj@pfri.hr}


\newtheorem{theorem}{Theorem}

\begin{document}
\maketitle

\begin{abstract}
In this paper our aim is to present some monotonicity and convexity properties for the one dimensional regularization of the Coulomb potential, which has applications in the study of atoms in magnetic fields and which is in fact a particular case of the Tricomi confluent hypergeometric function. Moreover, we present some Tur\'an type inequalities for the function in the question and we deduce from these inequalities some new tight upper bounds for the Mills ratio of the standard normal distribution.
\end{abstract}

\section{\bf Introduction}
\setcounter{equation}{0}

Consider the integral
$$V_q(x)=\frac{2e^{x^2}}{\Gamma(q+1)}\int_x^{\infty}e^{-t^2}(t^2-x^2)^qdt,$$
where $q>-1$ and $x>0.$ This integral can be regarded as the one dimensional regularization of the Coulomb potential, which has applications in the study of atoms in magnetic fields, see \cite{ruskai} for more details. Recently, Ruskai and Werner \cite{ruskai}, and later Alzer \cite{alzer} studied intensively the properties of this integral. In \cite{alzer,ruskai} the authors derived a number of monotonicity and convexity properties for the function $V_q,$ as well as many functional inequalities.

It is important to mention that $V_q$ in particular when $q=0$ becomes
$$m(x)=\frac{1}{\sqrt{2}}V_0\left(\frac{x}{\sqrt{2}}\right)=e^{x^2/2}\int_x^{\infty}e^{-t^2/2}dt,$$
which is the so-called Mills ratio of the standard normal distribution, and appears frequently in economics and statistics. See for example \cite{bariczmills} and the references therein for more details on this function.

The purpose of the present study is to make a contribution to the subject and to deduce some new monotonicity and convexity properties for the function $V_q,$
as well as some new functional inequalities. The paper is organized as follows. In section 2 we present the convexity results concerning the function $V_q$ together with some Tur\'an type inequalities. Note that the convexity results are presented in three equivalent formulations. Section 3 is devoted for concluding remarks. In this section we point out that $V_q$ is in fact a particular case of the Tricomi confluent hypergeometric function, and we deduce some other functional inequalities for $V_q.$ In this section we also point out that the Tur\'an type inequalities obtained in section 2 are particular cases of the recent results obtained by Baricz and Ismail \cite{ismail} for Tricomi confluent hypergeometric functions, however, the proofs are different. Finally, in section 3 we use the Tur\'an type inequalities for the function $V_q$ to derive some new tight upper bounds for the Mills ratio $m$ of the standard normal distribution.

\section{\bf Functional inequalities for the function $V_q$}
\setcounter{equation}{0}

The first main result of this paper is the following theorem. Parts {\bf a} and {\bf b} of this theorem are generalizations of parts {\bf b} and {\bf d} of \cite[Theorem 2.5]{bariczmills}.

\begin{theorem}\label{th1}
The next assertions are true:
\begin{enumerate}
\item[\bf a.] The function $x\mapsto xV_{q}'(x)/V_q(x)$ is strictly decreasing on $(0,\infty)$ for $q>-1.$
\item[\bf b.] The function $x\mapsto x^2V_{q}'(x)$ is strictly decreasing on $(0,\infty)$ for $q>-1.$
\item[\bf c.] The function $x\mapsto x^{-1}V_q'(x)$ is strictly increasing on $(0,\infty)$ for $q\geq0.$
\item[\bf d.] The function $x\mapsto V_{q}'(x)/(xV_q(x))$ is strictly increasing on $(0,\infty)$ for $q\geq 0.$
\end{enumerate}
\end{theorem}

\begin{proof}

{\bf a.} Observe that $V_q(x)$ can be rewritten as \cite[Lemma 1]{alzer}
$$V_q(x)=\frac{x^{q+1/2}}{\Gamma(q+1)}\int_0^{\infty}e^{-xs}\frac{s^q}{(x+s)^{1/2}}ds.$$
By using the change of variable $s=ux$ we obtain
\begin{equation}\label{A0}V_q(x)=\frac{x^{2q+1}}{\Gamma(q+1)}\int_0^{\infty}e^{-x^2u}\frac{u^q}{(1+u)^{1/2}}du,\end{equation}
and differentiating with respect to $x$ both sides of this relation we get
$$V_q'(x)=\frac{(2q+1)x^{2q}}{\Gamma(q+1)}\int_0^{\infty}e^{-x^2u}\frac{u^q}{(1+u)^{1/2}}du
-\frac{2x^{2q+2}}{\Gamma(q+1)}\int_0^{\infty}e^{-x^2u}\frac{u^{q+1}}{(1+u)^{1/2}}du.$$
Thus, for $q>-1$ and $x>0$ we obtain the differentiation formula
\begin{equation}\label{differ}
xV_q'(x)=(2q+1)V_q(x)-2(q+1)V_{q+1}(x),
\end{equation}
which in turn implies that
$$\frac{xV_q'(x)}{V_q(x)}=2q+1-2(q+1)\frac{V_{q+1}(x)}{V_q(x)}.$$
Now, recall that \cite[Theorem 7]{alzer} if $p>q>-1,$ then the function $x\mapsto V_p(x)/V_q(x)$ is strictly increasing
on $(0,\infty).$ In particular, the function $x\mapsto {V_{q+1}(x)}/{V_q(x)}$ is strictly increasing on $(0,\infty)$ for $q>-1,$ and by using the above relation we obtain that indeed the function $x\mapsto xV_{q}'(x)/V_q(x)$ is strictly decreasing on $(0,\infty)$ for $q>-1.$

{\bf b.} According to \cite[p. 429]{alzer} we have
\begin{equation}\label{deriv}V_{q}'(x)=-\frac{x}{\Gamma(q+1)}\int_0^{\infty}e^{-t}\frac{t^q}{(x^2+t)^{3/2}}dt.\end{equation}
Observe that for $q>-1$ and $x>0$ we have
\begin{align*}
\left[-\Gamma(q+1)xV_{q}'(x)\right]'&=\left[x^2\int_0^{\infty}e^{-t}\frac{t^q}{(x^2+t)^{3/2}}dt\right]'\\
&=\int_0^{\infty}e^{-t}\frac{xt^q}{(x^2+t)^{3/2}}\left(2-\frac{3x^2}{x^2+t}\right)dt\\
&>-x\int_0^{\infty}e^{-t}\frac{t^q}{(x^2+t)^{3/2}}dt=\Gamma(q+1)V_q'(x).
\end{align*}
In other words, we proved that for $x>0$ and $q>-1$ the differential inequality $$-(xV_{q}'(x))'>V_q'(x),$$ that is,
$$xV_q''(x)+2V_q'(x)<0$$ is valid. Consequently $$(x^2V_q'(x))'=x(2V_q'(x)+xV_q''(x))<0$$
for all $x>0$ and $q>-1,$ which means that indeed the function $x\mapsto x^2V_{q}'(x)$ is strictly decreasing on $(0,\infty)$ for $q>-1.$

{\bf c.} Recall the following differentiation formula \cite[p. 439]{ruskai}
\begin{equation}\label{difvq}
V_q'(x)=2x(V_q(x)-V_{q-1}(x)),
\end{equation}
which holds for $q\geq 0$ and $x>0.$ Here by convention $V_{-1}(x)=1/x,$ see \cite[p. 435]{ruskai}. On the other hand, it is known \cite[Theorem 7]{alzer} that
if $p>q>-1,$ then $x\mapsto V_p(x)-V_q(x)$ is strictly increasing on $(0,\infty).$ Consequently, $$x\mapsto x^{-1}V_q'(x)=2(V_q(x)-V_{q-1}(x))$$ is strictly increasing on $(0,\infty)$ for all $q\geq0.$

{\bf d.} Using again the fact that \cite[Theorem 7]{alzer} if $p>q>-1,$ then the function $x\mapsto V_p(x)/V_q(x)$ is strictly increasing
on $(0,\infty),$ we get that
$$x\mapsto \frac{V_q'(x)}{xV_q(x)}=2\left(1-\frac{V_{q-1}(x)}{V_q(x)}\right)$$
is strictly increasing on $(0,\infty)$ for all $q\geq0.$
\end{proof}

Now, we recall the definition of convex functions with respect to H\"older means or power means. For $a\in\mathbb{R},$ $\alpha\in[0,1]$ and $x,y>0$, the power mean $H_a$ of order $a$ is defined by
$$H_a(x,y)=\displaystyle\left\{\begin{array}{lll} \displaystyle\left({\alpha x^a+(1-\alpha)y^a}\right)^{1/a},&a\neq 0\\ 
x^{\alpha}y^{1-\alpha}\, , & a=0\end{array}\right..$$
We consider the continuous function $\varphi:I\subset(0,\infty)\to (0,\infty),$ and let $H_a(x,y)$
and $H_b(x,y)$ be the power means of order $a$ and $b$ of $x>0$ and $y>0$. For $a,b\in\mathbb{R}$ we say that $\varphi$ is $H_aH_b$-convex or just simply $(a,b)$-convex, if for $a,b\in\mathbb{R}$ and for all $x,y\in I$ we have
$$\varphi(H_a(x,y))\leq H_b(\varphi(x),\varphi(y)).$$
If the above inequality is reversed, then we say that $\varphi$ is $H_aH_b$-concave or simply $(a,b)$-concave. It is worth to note that $(1,1)$-convexity means the usual convexity, $(1,0)$ is the logarithmic convexity and $(0,0)$-convexity is the geometrical (or multiplicative) convexity. Moreover, we mention that if the function $f$ is differentiable, then (see \cite[Lemma 3]{bari}) it is $(a,b)$-convex (concave) if and only if $$x\mapsto x^{1-a}\varphi'(x)[\varphi(x)]^{b-1}$$ is increasing (decreasing).

For the sake of completeness we recall here also the definitions of log-convexity and geometrical convexity. A function $f \colon (0,\infty)\to(0,\infty)$ is said to be logarithmically convex, or simply log-convex, if its natural logarithm $\ln f$ is convex, that is, for all $x,y>0$ and $\lambda\in[0,1]$ we have
   $$f(\lambda x+(1-\lambda)y) \leq \left[f(x)\right]^{\lambda}\left[f(y)\right]^{1-\lambda}.$$
A similar characterization of log-concave functions also holds. By definition, a function $g \colon (0,\infty)\rightarrow(0,\infty)$ is said to be geometrically (or multiplicatively) convex if it is convex with respect to the geometric mean, that is, if for all $x,y>0$ and all $\lambda\in[0,1]$ the inequality
   $$g(x^{\lambda}y^{1-\lambda}) \leq[g(x)]^{\lambda}[g(y)]^{1-\lambda}$$
holds. The function $g$ is called geometrically concave if the above inequality is reversed. Observe that, actually the
geometrical convexity of a function $g$ means that the function $\ln g$ is a convex function of $\ln x$ in
the usual sense. We also note that the differentiable function $f$ is log-convex (log-concave) if and only if
$x \mapsto f'(x)/f(x)$ is increasing (decreasing), while the differentiable function $g$ is geometrically convex (concave) if
and only if the function $x \mapsto xg'(x)/g(x)$ is increasing (decreasing). 

The next result is a reformulation of Theorem \ref{th1} in terms of power means.

\begin{theorem}\label{th2}
The next assertions are true:
\begin{enumerate}
\item[\bf a.] $V_q$ is strictly $(0,0)$-concave on $(0,\infty)$ for $q>-1.$
\item[\bf b.] $V_q$ is strictly $(-1,1)$-concave on $(0,\infty)$ for $q>-1.$
\item[\bf c.] $V_q$ is strictly $(2,1)$-convex on $(0,\infty)$ for $q\geq0.$
\item[\bf d.] $V_q$ is strictly $(2,0)$-convex on $(0,\infty)$ for $q\geq 0.$
\end{enumerate}
In particular, for all $q\geq0$ and $x,y>0$ the next inequalities
\begin{equation}\label{power1}
V_q\left(\sqrt{\frac{x^2+y^2}{2}}\right)<\sqrt{V_q(x)V_q(y)}<V_q(\sqrt{xy})
\end{equation}
\begin{equation}\label{power2}
\frac{V_q(x)+V_q(y)}{2}<V_q\left(\frac{2xy}{x+y}\right)
\end{equation}
are valid. Moreover, the second inequality in \eqref{power1} is valid for all $q>-1,$ as well as the inequality \eqref{power2}.
In each of the above inequalities we have equality if and only if $x=y.$
\end{theorem}

Now, we extend some of the results of the above theorem to $(a,b)$-convexity with respect to power means. We note that in the proof of the next theorem we used the corresponding results of Theorem \ref{th1}. Moreover, it is easy to see that parts {\bf a}, {\bf b}, {\bf c} and {\bf d} of Theorem \ref{th3} in particular reduce to the corresponding parts of Theorem \ref{th1}. Thus, in fact the corresponding parts of Theorem \ref{th1} and \ref{th3} are equivalent.

\begin{theorem}\label{th3}
The following assertions are true:
\begin{enumerate}
\item[\bf a.] $V_q$ is strictly $(a,b)$-concave on $(0,\infty)$ for $a,b\leq 0$ and $q>-1.$
\item[\bf b.] $V_q$ is strictly $(a,b)$-concave on $(0,\infty)$ for $b\leq 1$ and $q>-1\geq a.$
\item[\bf c.] $V_q$ is strictly $(a,b)$-convex on $(0,\infty)$ for $a\geq 2,$ $b\geq1$ and $q\geq0.$
\item[\bf d.] $V_q$ is strictly $(a,b)$-convex on $(0,\infty)$ for $a\geq 2,$ $b\geq0$ and $q\geq 0.$
\item[\bf e.] $V_q$ is strictly $(a,b)$-concave on $(0,\infty)$ for $a\leq 1,$ $b\leq -1$ and $q\geq0.$
\end{enumerate}
\end{theorem}

\begin{proof}
{\bf a.} We consider the functions $u_1,v_1,w_1:(0,\infty)\to\mathbb{R},$ which are defined by
$$u_1(x)=\frac{xV_q'(x)}{V_q(x)},\ \ \ v_1(x)=x^{-a},\ \ \ w_1(x)=V_q^b(x).$$
For $a,b\leq0$ and $q>-1$ the functions $v_1$ and $w_1$ are increasing on $(0,\infty),$ and by using part {\bf a} of Theorem \ref{th1}, we obtain that
the function $$x\mapsto M_q(x)=u_1(x)v_1(x)w_1(x)=x^{1-a}V_q'(x)V_q^{b-1}(x)$$
is strictly decreasing on $(0,\infty).$ Here we used that $u_1(x)<0$ for all $x>0$ and $q>-1.$ According to \cite[Lemma 3]{bari} we obtain that indeed the function $V_q$ is strictly $(a,b)$-concave on $(0,\infty)$ for $a,b\leq 0$ and $q>-1.$

{\bf b.} Similarly, if we consider the functions $u_2,v_2,w_2:(0,\infty)\to\mathbb{R},$ defined by
$$u_2(x)=x^{-a-1},\ \ \ v_2(x)=x^2V_q'(x),\ \ \ w_2(x)=V_q^{b-1}(x),$$
then for $a\leq -1<q$ and $b\leq 1$ the function $$x\mapsto M_q(x)=u_2(x)v_2(x)w_2(x)=x^{1-a}V_q'(x)V_q^{b-1}(x)$$
is strictly decreasing on $(0,\infty).$ Here we used part {\bf b} of Theorem \ref{th1}.

{\bf c.} Analogously, if we consider the functions $u_3,v_3,w_3:(0,\infty)\to\mathbb{R},$ defined by
$$u_3(x)=x^{2-a},\ \ \ v_3(x)=x^{-1}V_q'(x),\ \ \ w_3(x)=V_q^{b-1}(x),$$
then for $a\geq2,$ $b\geq 1$ and $q\geq0$ the function $$x\mapsto M_q(x)=u_3(x)v_3(x)w_3(x)=x^{1-a}V_q'(x)V_q^{b-1}(x)$$
is strictly increasing on $(0,\infty).$ Here we used part {\bf c} of Theorem \ref{th1}.

{\bf d.} If we consider the functions $u_4,v_4,w_4:(0,\infty)\to\mathbb{R},$ defined by
$$u_4(x)=x^{2-a},\ \ \ v_4(x)=x^{-1}V_q'(x)V_q^{-1}(x),\ \ \ w_4(x)=V_q^{b}(x),$$
then for $a\geq 2,$ $b\leq 1,$ $q\geq0,$ the function $$x\mapsto M_q(x)=u_4(x)v_4(x)w_4(x)=x^{1-a}V_q'(x)V_q^{b-1}(x)$$
is strictly increasing on $(0,\infty).$ Here we used part {\bf d} of Theorem \ref{th1}.

{\bf e.} If we consider the functions $u_4,v_4,w_4:(0,\infty)\to\mathbb{R},$ defined by
$$u_5(x)=x^{1-a},\ \ \ v_5(x)=V_q'(x)V_q^{-2}(x),\ \ \ w_5(x)=V_q^{b+1}(x),$$
then for $a\leq 1,$ $b\leq -1,$ $q\geq0,$ the function $$x\mapsto M_q(x)=u_5(x)v_5(x)w_5(x)=x^{1-a}V_q'(x)V_q^{b-1}(x)$$
is strictly decreasing on $(0,\infty).$ Here we used the fact that for $q\geq 0$ the function $1/V_q$ is strictly convex (see \cite[Theorem 2]{alzer}) on $(0,\infty),$ which is equivalent to the fact that $V_q$ is strictly $(1,-1)$-concave on $(0,\infty)$, or to that the function $v_5$ is strictly decreasing on $(0,\infty).$
\end{proof}

The following theorem presents some Tur\'an type inequalities for the function $V_q$. These kind of inequalities are named after the Hungarian mathematician Paul Tur\'an who proved a similar inequality for Legendre polynomials. For more details on Tur\'an type inequalities we refer to the papers \cite{bariczD,ismail} and to the references therein.

\begin{theorem}
For $x>0$ the function $q\mapsto \Gamma(q+1)V_q(x)$ is strictly log-convex on $(-1,\infty),$ and if $q>-1/2$ and $x>0,$ then the next Tur\'an type inequalities hold
\begin{equation}\label{turan}
\frac{(q+2)(2q+1)}{(q+1)(2q+3)}V_q(x)V_{q+2}(x)<V_{q+1}^2(x)<\frac{q+2}{q+1}V_q(x)V_{q+2}(x).
\end{equation}
Moreover, the right-hand side of \eqref{turan} is valid for $q>-1$ and $x>0.$ The left-hand side of \eqref{turan} is sharp as $x$ tends to $0.$
\end{theorem}

\begin{proof}
We use the notation $f(q)=\Gamma(q+1)V_q(x).$ Since \cite[p. 426]{alzer}
$$V_q(x)=\frac{1}{\Gamma(q+1)}\int_0^{\infty}e^{-t}\frac{t^q}{(x^2+t)^{1/2}}dt$$
it follows that
$$f(q)=\int_0^{\infty}e^{-t}\frac{t^q}{(x^2+t)^{1/2}}dt.$$ By using the H\"older-Rogers inequality for integrals we obtain that for all $q_1,q_2>-1,$ $q_1\neq q_2,$ $\alpha\in(0,1)$ and $x>0$ we have
\begin{align*}f(\alpha q_1+(1-\alpha)q_2)&=\int_0^{\infty}e^{-t}\frac{t^{\alpha q_1+(1-\alpha)q_2}}{(x^2+t)^{1/2}}dt\\
&=\int_0^{\infty}\left(e^{-t}\frac{t^{q_1}}{(x^2+t)^{1/2}}\right)^{\alpha}\left(e^{-t}\frac{t^{q_2}}{(x^2+t)^{1/2}}\right)^{1-\alpha}dt\\
&<\left(\int_0^{\infty}e^{-t}\frac{t^{q_1}}{(x^2+t)^{1/2}}dt\right)^{\alpha}\left(\int_0^{\infty}e^{-t}\frac{t^{q_2}}{(x^2+t)^{1/2}}dt\right)^{1-\alpha}\\
&=(f(q_1))^{\alpha}(f(q_2))^{1-\alpha},\end{align*}
that is, the function $f$ is strictly log-convex on $(-1,\infty)$ for $x>0.$ Now, choosing $\alpha=1/2,$ $q_1=q$ and $q_2=q+2$ in the above inequality we obtain the Tur\'an type inequality $$f^2(q+1)<f(q)f(q+2)$$ which is equivalent to the inequality
$$V_{q+1}^2(x)<\frac{\Gamma(q+3)\Gamma(q+1)}{\Gamma^2(q+2)}V_q(x)V_{q+2}(x),$$
valid for $q>-1$ and $x>0.$ After simplifications we get the right-hand side of \eqref{turan}.

Now, we focus on the left-hand side of \eqref{turan}. First observe that from \eqref{deriv} it follows that $V_q'(x)<0$ for all $x>0$ and $q>-1.$ In view of the differentiation formula \eqref{differ} this implies that for $x>0$ and $q>-1$ we have
\begin{equation}\label{bound}(2q+1)V_q(x)<2(q+1)V_{q+1}(x).\end{equation}
On the other hand, recall that the function $x\mapsto V_{q+1}(x)/V_q(x)$ is strictly increasing on $(0,\infty)$ for $q>-1,$ that is, the inequality
$$\left(V_{q+1}(x)/V_q(x)\right)'>0$$ is valid for $x>0$ and $q>-1.$ By using \eqref{differ} it can be shown that the above assertion is equivalent to the Tur\'an type inequality
\begin{equation}\label{turan2}(q+1)V_{q+1}^2(x)-(q+2)V_q(x)V_{q+2}(x)>-V_q(x)V_{q+1}(x),\end{equation}
where $x>0$ and $q>-1.$ Combining \eqref{bound} with \eqref{turan2} for $q>-1/2$ and $x>0$ we have
$$(q+1)V_{q+1}^2(x)-(q+2)V_q(x)V_{q+2}(x)>-\frac{2(q+1)}{2q+1}V_{q+1}^2(x),$$
which is equivalent to the left-hand side of \eqref{turan}.

Finally, since
$$V_q(0)=\frac{\Gamma(q+1/2)}{\Gamma(q+1)},$$
it follows that
$$\frac{V_{q+1}^2(0)}{V_q(0)V_{q+2}(0)}=\frac{(q+2)(2q+1)}{(q+1)(2q+3)},$$
and thus indeed the left-hand side of \eqref{turan} is sharp as $x$ tends to $0.$
\end{proof}

\section{\bf Concluding remarks and further results}
\setcounter{equation}{0}

\subsection{Connection with Tricomi confluent hypergeometric functions and Tur\'an type inequalities}

First consider the Tricomi confluent hypergeometric function, called also sometimes as the
confluent hypergeometric function of the second kind,
$\psi(a,c,\cdot),$ which is a particular solution of the so-called confluent
hypergeometric  differential equation $$xw''(x)+(c-x)w'(x)-aw(x)=0$$ and its value is defined in terms of the usual Kummer confluent
hypergeometric function $\Phi(a,c,\cdot)$ as
$$\psi(a,c,x)=\frac{\Gamma(1-c)}{\Gamma(a-c+1)}\Phi(a,c,x)+
\frac{\Gamma(c-1)}{\Gamma(a)}x^{1-c}\Phi(a-c+1,2-c,x).$$ For $a,x>0$ this function possesses the integral representation
   \[ \psi(a,c,x) = \frac{1}{\Gamma(a)} \int_0^\infty e^{-xt}t^{a-1}(1+t)^{c-a-1}dt,\]
   and consequently we have
   \begin{equation}\label{vqdef} V_q(x) = \frac{x^{2q+1}}{\Gamma(q+1)} \int_0^\infty {\rm e}^{-x^2u}\, \frac{u^q}{\sqrt{1+u}}\, {\rm d}u
             = x^{2q+1} \psi(q+1, q+3/2, x^2) \, .\end{equation}
Thus, the Tur\'an type inequality \eqref{turan} can be rewritten as
\begin{equation}\label{turan3}
\frac{(a+1)(2a-1)}{a(2a+1)}<\frac{\psi^2(a+1,a+3/2,x)}{\psi(a,a+1/2,x)\psi(a+2,a+5/2,x)}<
\frac{a+1}{a},
\end{equation}
where $a>1/2$ and $x>0$ on the left-hand side, and $a>0$ and $x>0$ on the right-hand side. Now, applying the Kummer transformation
   \[ \psi(a,c,x) = x^{1-c}\psi(1+a-c,2-c,x),\]
the above Tur\'an type inequality becomes
\begin{equation}\label{turan3}
\frac{c(2c-3)}{(c-1)(2c-1)}<\frac{\psi^2(1/2,c,x)}{\psi(1/2,c-1,x)\psi(1/2,c+1,x)}<
\frac{2c-3}{2c-1},
\end{equation}
where $x>0>c$ on the left-hand side, and $c<1/2$ and $x>0$ on the right-hand side. It is important to mention here that the right-hand side of \eqref{turan3} is not sharp when $c<0$. Namely, in \cite[Theorem 4]{ismail} it was proved that
the sharp Tur\'an type inequality
$$\psi^2(a,c,x)-\psi(a,c-1,x)\psi(a,c+1,x)<0$$
is valid for $a>0>c$ and $x>0$ or $a>c-1>0$ and $x>0.$ This implies that
$$\frac{\psi^2(1/2,c,x)}{\psi(1/2,c-1,x)\psi(1/2,c+1,x)}<1$$
holds for $c<0$ and $x>0$ or $c\in(1,3/2)$ and $x>0,$ and the constant $1$ on the right-hand side of this inequality is the best possible. The above Tur\'an type inequality clearly improves the right-hand side of \eqref{turan3} when $c<0$, and this means that for $q>-1$ and $x>0$ the right-hand side of \eqref{turan} can be improved as follows
\begin{equation}\label{turan4}V_{q+1}^2(x)<V_{q}(x)V_{q+2}(x).\end{equation}
Note also that very recently Baricz and Ismail in \cite[Theorem 4]{ismail} proved
the sharp Tur\'an type inequality
$$\frac{a}{c(a-c+1)}\psi^2(a,c,x)<\psi^2(a,c,x)-\psi(a,c-1,x)\psi(a,c+1,x),$$
which is valid for $a>0>c$ and $x>0.$ This inequality can be rewritten as
$$\frac{c(a-c+1)}{(c-1)(a-c)}<\frac{\psi^2(a,c,x)}{\psi(a,c-1,x)\psi(a,c+1,x)},$$
which for $a=1/2$ reduces to the left-hand side of \eqref{turan3}. It is important to note here that according to \cite[Theorem 4]{ismail}
in the above Tur\'an type inequalities the constants $$a(c(a-c+1))^{-1}\ \ \ \mbox{and}\ \ \ (c(a-c+1))/((c-1)(a-c))^{-1}$$ are best possible, and so is the constant
$${c(2c-3)}/((c-1)(2c-1))^{-1}$$ in \eqref{turan3}.

We also mention that the method of proving \eqref{turan} is completely different than of the proof of \cite[Theorem 4]{ismail}. Note also that the sharp Tur\'an type inequality \eqref{turan4} is in fact related to the following open problem \cite[p. 87]{bariczD}: is the function $q\mapsto V_q(x)$ log-convex on $(-1,\infty)$ for $x>0$ fixed? If this result were be true then would improve Alzer's result \cite[Theorem 3]{alzer}, which states that the function $q\mapsto V_q(x)$ is convex on $(-1,\infty)$ for all $x>0$ fixed.

Recently, for $x>0$ Simon \cite{simon} proved the next Tur\'an type inequalities
\begin{equation}\label{turan5}
\psi(a-1,c-1,x)\psi(a+1,c+1,x)-\psi^2(a,c,x)\leq \frac{1}{x}\psi^2(a,c,x)\psi(a+1,c+1,x),
\end{equation}
\begin{equation}\label{turan6}
\psi(a,c-1,x)\psi(a,c+1,x)-\psi^2(a,c,x)\leq \frac{1}{x}\psi(a,c,x)\psi(a,c-1,x).
\end{equation}
In \eqref{turan5} it is supposed that $a>1$ and $c<a+1,$ while in \eqref{turan6} it is assumed that $a\geq 1$ or $a>0$ and $c\leq a+2.$
By using \eqref{vqdef} the inequality \eqref{turan5} in particular reduces to
$$V_q(x)V_{q+2}(x)\leq V_{q+1}^2(x)\left(1+x^{-2(q+3)}V_{q+2}(x)\right),$$
where $q>-1$ and $x>0.$ Now, observe that by using the above mentioned Kummer transformation in \eqref{vqdef} we obtain
$$V_q(x)=\psi(1/2,1/2-q,x^2),$$ and by using this, \eqref{turan6} in particular reduces to
$$V_{q}(x)V_{q+2}(x)-V_{q+1}^2(x)\leq \frac{1}{x}V_{q+1}(x)V_{q+2}(x),$$
where $q>-1$ and $x>0.$ Combining this inequality with \eqref{turan4} for $q>-1$ and $x>0$ we obtain
$$-\frac{1}{x}V_{q+1}(x)V_{q+2}(x)\leq V_{q+1}^2(x)-V_{q}(x)V_{q+2}(x)<0.$$

\subsection{Connection with Mills ratio and some new bounds for this function}

In this subsection we would like to show that the inequalities presented above for the function $V_q$ can be used
to obtain many new results for the Mills ratio $m.$ For this, first recall that Mills' ratio $m$ satisfies the differential
equation \cite[p. 1365]{bariczmills} $m'(x)=xm(x)-1$ and hence
\begin{equation}\label{vodiffer}V_0'(x)=(\sqrt{2}\cdot m(x\sqrt{2}))'=2(\sqrt{2}x\cdot m(x\sqrt{2})-1)=2(xV_0(x)-1).\end{equation}
Note that this differentiation formula can be deduced also from \eqref{difvq}. Observe that by using \eqref{differ} and \eqref{vodiffer} we get
\begin{equation}\label{v0v1}2V_1(x)=(1-2x^2)V_0(x)+2x,\end{equation}
\begin{align*}8V_2(x)=(4x^4-4x^2+3)V_0(x)+2x(3-2x^2).\end{align*}

Now, if $q\to-1$ in \eqref{turan4}, then we get that the Tur\'an type inequality $$V_0^2(x)<V_{-1}(x)V_1(x)$$ is valid for $x>0,$ and this is equivalent to
$$2xV_0^2(x)+(2x^2-1)V_0(x)-2x<0.$$ From this we obtain that for $x>0$ the inequality
$$V_0(x)<\frac{1-2x^2+\sqrt{4x^4+12x^2+1}}{4x}$$
is valid, and rewriting in terms of Mills ratio we get
$$m(x)<\frac{1-x^2+\sqrt{x^4+6x^2+1}}{4x}.$$

Similarly, if we take $q=0$ in the left-hand side of \eqref{turan}, then we get
$$2V_0(x)V_2(x)<3V_1^2(x),$$ which can be rewritten as
$$4x(x^2-1)V_0^2(x)+(3-10x^2)V_0(x)+6x>0.$$
From this for $x>0$ we obtain
$$V_0(x)<\frac{10x^2-3-\sqrt{4x^4+36x^2+9}}{8x(x^2-1)},$$
which in terms of Mills ratio can be rewritten as
$$m(x)<\frac{5x^2-3-\sqrt{x^4+18x^2+9}}{4x(x^2-2)},$$
where $x>0.$ As far as we know these upper bounds on Mills ratio $m$ are new. We note that many other results of this kind can be obtained by using for example \eqref{turan4} for $q=0$ or by using the other Tur\'an type inequalities in the previous subsection.

Finally, we mention that if we take in \eqref{difvq} the value $q=0$ and we take into account that $V_0$ is strictly decreasing on $(0,\infty),$ we get the inequality $V_0(x)<1/x,$ which in terms of Mills ratio can be rewritten as $m(x)<1/x.$ This inequality is the well-known Gordon inequality for Mills' ratio, see \cite{gordon} for more details. Note that the inequality $V_0(x)<1/x$ can be obtained also from \eqref{bound}, just choosing $q=0$ and taking into account the relation \eqref{v0v1} between $V_0$ and $V_1.$ It is important to mention here that Gordon's inequality $m(x)<1/x$ is in fact a particular Tur\'an type inequality for the parabolic cylinder function, see \cite[p. 199]{ismail} for more details.

\subsection{Other results on Mills ratio and their generalizations}

It is worth to mention that it is possible to derive other inequalities for $V_q$ and its particular case $m$ by using the recurrence relations for this function. Namely, from \eqref{difvq} we get that $$V_q(x)<V_{q-1}(x)$$ for $q\geq 0$ and $x>0,$ and by using \eqref{differ} we obtain
$$(xV_q(x))'=2(q+1)(V_q(x)-V_{q+1}(x))>0,$$
where $x>0$ and $q>-1.$ On the other hand, by using \eqref{difvq} it follows
$$(xV_q(x))'=(2x^2+1)V_q(x)-2x^2V_{q-1}(x),$$
and from the previous inequality we get the inequality
\begin{equation}\label{boundquo}\frac{V_q(x)}{V_{q-1}(x)}>\frac{2x^2}{2x^2+1},\end{equation}
which holds for all $q\geq0$ and $x>0.$ Now, if we take $q=0$ and $q=1$ in \eqref{boundquo} we obtain the inequalities
$$\frac{2x}{2x^2+1}<V_0(x)<\frac{2x(2x^2+1)}{4x^4+4x^2-1},$$
where $x>0$ on the left-hand side, and $x\sqrt{2}>\sqrt{\sqrt{2}-1}$ on the right-hand side. This inequality in terms of Mills ratio can be rewritten as
\begin{equation}\label{boundmills}\frac{x}{x^2+1}<m(x)<\frac{x(x^2+1)}{x^4+2x^2-1},\end{equation}
where $x>0$ on the left-hand side and $x>\sqrt{\sqrt{2}-1}$ on the right-hand side. Observe that the right-hand side of \eqref{boundmills} is better than
Gordon's inequality $m(x)<1/x$ when $x>1.$ We also note that the left-hand side of \eqref{boundmills} is known and it was deduced by Gordon \cite{gordon}.

Now, let us consider the functions $f_1,f_2,f_3,f_4,f_5:(0,\infty)\to\mathbb{R},$ defined by
$$f_1(x)=\frac{x}{x^2+1},\ f_2(x)=\frac{1}{x},\ f_3(x)=\frac{x(x^2+1)}{x^4+2x^2-1},$$
$$f_4(x)=\frac{1-x^2+\sqrt{x^4+6x^2+1}}{4x},\ f_5(x)=\frac{5x^2-3-\sqrt{x^4+18x^2+9}}{4x(x^2-2)}.$$
Figure \ref{fig} shows that the above new upper bounds (for the Mills ratio of the standard normal distribution) are quite tight.

\begin{figure}[!ht]
   \centering
       \includegraphics[width=12cm]{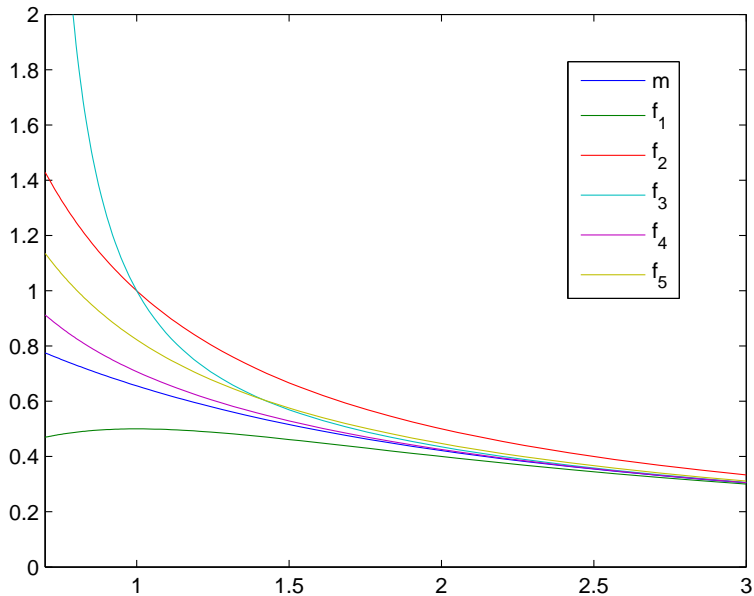}
       \caption{The graph of Mills' ratio $m$ of the standard normal distribution and of the bounds $f_1,$ $f_2,$ $f_3,$ $f_4$ and $f_5$ on $[0.7,3].$}
       \label{fig}
\end{figure}

\subsection{Connection with Gaussian hypergeometric functions}

Here we would like to show that $V_q$ can be expressed in terms of Gaussian hypergeometric functions. For this, we consider the following integral \cite[p. 18, Eq. 2.29]{Oberh}
   \begin{align*}
      \int_0^\infty \frac{x^{p-1}\, {\rm d}x}{(c+bx)^\nu\,(a+dx)^\mu} &= d^{-\mu}c^{-\nu+p}b^{-p}\, {\rm B}(p, \mu+\nu-p)\,
                         {}_2F_1\left( \mu, p; \mu+\nu; 1-\frac{ac}{bd}\right)\\
                      &= d^{-\mu+p}c^{-\nu}a^{-p}\, {\rm B}(p, \mu+\nu-p)\,
                         {}_2F_1\left( \nu, p; \mu+\nu; 1-\frac{bd}{ac}\right)\, ,
   \end{align*}
valid in both cases for all $0<p < \mu+\nu$. Specifying inside
   \[ a=c=d=1, \quad b=\frac{x^2}n, \quad \nu=n, \mu=\frac12, \quad p=q+1\, ,\]
we conclude
   \begin{align*}
      V_q(x) &= \dfrac{x^{2q+1}}{\Gamma(q+1)} \int_0^\infty \lim_{n \to \infty}\left( 1+ \frac{x^2\,t}n\right)^{-n}\,
                \frac{t^q\, {\rm d}t}{\sqrt{1+t}} \\
             &= \frac1x\, \lim_{n \to \infty} n^{q+1} \frac{\Gamma(n+\frac12-q)}{\Gamma(n+\frac32)}\,
                {}_2F_1\left( \frac12, q+1; n+\frac12; 1-\frac n{x^2}\right) \\
             &= \dfrac{x^{2q+1}}{\Gamma(q+1)} \lim_{n \to \infty} \frac{\Gamma(n+\frac12-q)}{\Gamma(n+\frac32)}\,
                {}_2F_1\left( n, q+1; n+\frac12; 1-\frac{x^2}n\right)\, .
   \end{align*}

\subsection{Lower and upper bounds for the function $V_q$} It is of considerable interest to find lower and upper bounds for the function $x \mapsto V_q(x)$ itself. Therefore remarking the obvious inequality $1+a \le {\rm e}^a,$ $a \in \mathbb R$, we conclude the following. Having in mind the integral expression \eqref{A0}, and specifying $a=u$, we get
   \begin{align*} \label{C0}
      V_q(x) &= \frac{x^{2q+1}}{\Gamma(q+1)} \int_0^\infty {\rm e}^{-x^2u}\, \frac{u^q}{\sqrt{1+u}}\, {\rm d}u\nonumber \\
             &\ge \frac{x^{2q+1}}{\Gamma(q+1)} \int_0^\infty {\rm e}^{-(x^2 + \frac12) u}\, u^q\, {\rm d}u \nonumber \\
             &= \frac{2^{q+1}\,x^{2q+1}}{(1+2x^2)^{q+1}}\, .\bigskip
   \end{align*}
Similarly, transforming the integrand of \eqref{A0} by the arithmetic mean--geometric mean inequality
$1+u \ge 2\sqrt{u},$ $u\ge0$, we get
   \[ V_q(x) \le \frac{x^{2q+1}}{\sqrt{2}\,\Gamma(q+1)} \int_0^\infty {\rm e}^{-x^2u}\, u^{q-\frac14}\, {\rm d}u
             = \frac{\Gamma(q+\frac34)}{\sqrt{2x}\,\Gamma(q+1)} \,. \]
Finally, choosing $a=x^2t^{-1}$, we get
   \begin{align*}
      V_q(x) &= \frac1{\Gamma(q+1)} \int_0^\infty {\rm e}^{-u}\, \frac{u^q}{\sqrt{x^2+u}}\, {\rm d}u\, \\
             &\ge \frac1{\Gamma(q+1)} \int_0^\infty {\rm e}^{-u- \frac{x^2}{2u}}\, u^{q-\frac12}\, {\rm d}u \\
             &= \frac1{\Gamma(q+1)} \, Z_1^{q+\frac12}\left( \frac{x^2}2\right)\, ,
    \end{align*}
where
   \[ Z_{\rho}^{\nu}(t)=\int_0^{\infty}u^{\nu-1}{\rm e}^{-u^{\rho}-\frac{t}{u}}\,{\rm d}u\,,\]
stands for the so-called Kr\"atzel function, see \cite{kratzel}, and also \cite{BJP1}. Note that further consequent inequalities have been
established for the Kr\"atzel function in \cite{BJP1}, compare \cite[Theorem 1]{BJP1}.

\end{document}